\documentclass[12pt]{article}

\usepackage{pgfplots}
\usepackage{times}
\usepackage{amsmath,amsfonts, amstext,amssymb,amsbsy,amsopn,amsthm}
\usepackage{dsfont}
\usepackage{esint}
\usepackage{graphicx}   
\usepackage{hyperref}
\usepackage[all]{xy}
\usepackage{color}
\usepackage{tikz}
\usepackage{lineno,hyperref}

\newtheorem{definition}{Definition}[section]
\newtheorem{theorem}[definition]{Theorem}

\theoremstyle{remark}
\newtheorem{remark}[definition]{Remark}
\numberwithin{equation}{section}

\newcommand{\abs}[1]{\lvert#1\rvert}

\newcommand{\R}{\mathbb{R}}

\newcommand{\lap}{\mbox{$\triangle$}}

\newcommand{\fr}{\displaystyle\frac}
\newcommand{\jf}{\displaystyle\int}
\newcommand{\lt}{\left}
\newcommand{\rt}{\right}
\newcommand{\lm}{\lambda}

\newcommand{\mb}{\mbox}

\newcommand{\tm}{\times}
\newcommand{\ve}{\varepsilon}

\newcommand{\be}{\begin{equation}}
\newcommand{\ee}{\end{equation}}
\newcommand{\bee}{\begin{equation*}}
\newcommand{\eee}{\end{equation*}}

\title{  Master equations with indefinite nonlinearities}

\author{Wenxiong Chen and Yahong Guo}

\usepackage{hyperref}
\begin{document}
\maketitle

\begin{abstract}

In this paper, we  consider the  following indefinite fully fractional heat equation involving the master operator
\begin{equation*}
(\partial_t -\Delta)^{s} u(x,t) = x_1u^p(x,t)\ \ \mbox{in}\ \R^n\times\R ,
\end{equation*}
where $s\in(0,1)$, and  $-\infty < p < \infty$. Under mild conditions, we prove that there is no positive bounded solutions. To this end,
we first show that the solutions are strictly increasing along $x_1$ direction by employing the direct method of moving planes. Then by constructing an unbounded sub-solution, we derive the nonexistence of bounded solutions.

To circumvent the difficulties caused by the fully fractional master operator, we introduced some new ideas and novel approaches that, as we believe, will become useful tool in studying a variety of other fractional elliptic and parabolic problems.

 \end{abstract}

 \bigskip

\bigskip


\section{Introduction}
\

The primary objective of this paper is to prove the non-existence of positive bounded solutions for the following master equations with indefinite nonlinearity
\begin{equation}\label{2.0}
(\partial_t -\Delta)^{s} u(x,t) =  x_1u^p(x,t)\ \ \mbox{in}\ \ \R^n\times\R.
\end{equation}
Here the fully fractional heat operator $(\partial_t-\Delta)^s$ was initially introduced by M. Riesz in \cite{Riesz}. It is a nonlocal pseudo differential operator of order $2s$ in space variables and of order $s$ in the time variable and can be defined by the following singular integral
\begin{equation}\label{nonlocaloper}
(\partial_t-\Delta)^s u(x,t)
:=C_{n,s}\int_{-\infty}^{t}\int_{\mathbb{R}^n}
  \frac{u(x,t)-u(y,\tau)}{(t-\tau)^{\frac{n}{2}+1+s}}e^{-\frac{|x-y|^2}{4(t-\tau)}}\operatorname{d}\!y\operatorname{d}\!\tau,
\end{equation}
where $0<s<1$, the integral in $y$ is in the sense of Cauchy principal value, and
the normalization constant $$C_{n,s}=\frac{1}{(4\pi)^{\frac{n}{2}}|\Gamma(-s)|}$$
with $\Gamma(\cdot)$ denoting the Gamma function. Note that this operator is nonlocal both in space and time, since the value of $(\partial_t-\Delta)^s u$ at a given point $(x,t)$ depends on the values of $u$ in the whole space $\mathbb{R}^n$ and on  all the past time before $t$.

We say that $u$ is a classical entire solution of \eqref{2.0} if
 $$u(x,t)\in C^{2s+\epsilon,s+\epsilon}_{x,\, t,\, {\rm loc}}(\mathbb{R}^n\times\mathbb{R}) \cap \mathcal{L}(\mathbb{R}^n\times\mathbb{R})$$
for some $\varepsilon >0$,
which ensure that the singular integral in \eqref{nonlocaloper} is well defined. Here,
the slowly increasing function space $\mathcal{L}(\mathbb{R}^n\times\mathbb{R})$ is defined by
$$ \mathcal{L}(\mathbb{R}^n\times\mathbb{R}):=\left\{u(x,t) \in L^1_{\rm loc} (\mathbb{R}^n\times\mathbb{R}) \mid \int_{-\infty}^t \int_{\mathbb{R}^n} \frac{|u(x,\tau)|e^{-\frac{|x|^2}{4(t-\tau)}}}{1+(t-\tau)^{\frac{n}{2}+1+s}}\operatorname{d}\!x\operatorname{d}\!\tau<\infty,\,\, \forall \,t\in\mathbb{R}\right\}.$$
The definition of the local parabolic H\"{o}lder space $C^{2s+\epsilon,s+\epsilon}_{x,\, t,\, {\rm loc}}(\mathbb{R}^n\times\mathbb{R})$ will be specified in Section \ref{5}\,.

It is interesting to note that the fractional powers of heat operator $(\partial_t-\Delta)^s$
is reduced to the regular heat operator $\partial_t-\Delta$ as $s\rightarrow 1$ (cf. \cite{FNW}). Moreover,
when the space-time nonlocal operator $(\partial_t-\Delta)^s$ is applied to a function that depends only on either space or on time, it reduces  to a well-known fractional order operator (cf. \cite{ST}).

In particular, if $u$ is only a function of $x$, then
 \begin{equation*}
   (\partial_t-\Delta)^s u(x)=(-\Delta)^s u(x),
 \end{equation*}
where $(-\Delta)^s$ is the widely recognized fractional Laplacian. This operator holds significant interest due to its diverse applications across various scientific disciplines, including  physics, chemistry, and biology, such as in anomalous diffusion, quasi-geostrophic flows,
thin obstacle problem, phase transitions, crystal dislocation, flame propagation, conservation
laws, multiple scattering, minimal surfaces, optimization, turbulence models, water waves,
molecular dynamics, and image processing ( see \cite{AB, BG, CV, GO} and the references therein).
Additionally, these operators play crucial roles in probability
and finance \cite{Be} \cite{CT} \cite{RSM}. In particular, the fractional Laplacians can be interpreted as the
infinitesimal generator of a stable L\'{e}vy process \cite{Be}.

In recent decades, considerable attention has been dedicated to the analysis of solutions to fractional elliptic equations and a series of fruitful results have been obtained.  Interested readers can refer to \cite{CLL1, CLZ, CW1, CZhu, DLL, LLW, LZ} and references therein.

While if $u=u(t)$, then
 \begin{equation*}
   (\partial_t-\Delta)^s u(t)=\partial_t^s u(t),
 \end{equation*}
where $\partial_t^s$ is the Marchaud fractional derivative of order $s$, defined as \begin{equation}\label{1.00}
\partial^s_t u(x,t)=C_s \jf_{-\infty}^t\fr{u(x,t)-u(x,\tau)}{(t-\tau)^{1+s}}d\tau.
\end{equation}
It emerges in a variety of physical phenomena, for instance,
particle systems with sticking and trapping phenomena, magneto-thermoelastic heat conduction, plasma turbulence and so on (cf. \cite{ACV,ACV1, DCL1, DCL2, EE}).

The space-time nonlocal equation represented by \eqref{2.0} arises in various physical and biological phenomena, such as anomalous diffusion \cite{KBS}, chaotic dynamics \cite{Z}, biological invasions \cite{BRR} and so on. In the financial domain, it also serves as a valuable tool for modeling scenarios where the waiting time between transactions is correlated with ensuing price jumps (cf. \cite{RSM}).

One prominent application of the master equation \eqref{2.0} is in the representation of continuous-time random walks, where $u$ signifies the distribution of particles subject to random jumps occurring simultaneously with random time lags (cf. \cite{MK}). This model, serving as a generalization of Brownian random walks formulated with a local time derivative, characterizes particles undergoing uncorrelated random displacements at fixed time intervals. The introduction of time non-locality captures the influence of anomalously large waiting times on the dynamics, while space non-locality accommodates the existence of anomalously large jumps, such as L\'{e}vy flights connecting distant regions in space.


\bigskip

\subsection{ The background on indefinite nonlinearities} \,

\subsubsection{Local elliptic equations}
\,

First consider the indefinite problem for the regular Laplacian in a smooth bounded domain $\Omega$ in $\mathbb
R^n$.
\begin{equation}
\left \{
\begin{array}{rll}
-\lap u= &a(x)u^p \quad \quad &\mbox{in} \ \Omega, \medskip \\
u=&0 \quad \quad &\mbox{on} \ \partial\Omega. \label{semi0}
\end{array}
\right.
\end{equation}
Here $a(x)$ is a smooth function that changes signs in $\Omega$, hence we call the right hand side $a(x)u^p$
an {\em indefinite nonlinearity}.
Let
$$ \Omega^+:=\{x\in \Omega: a(x)>0\} \quad \mbox{and} \quad
\Omega^-:=\{x\in \Omega: a(x)<0\}, $$ and assume
$$ \Gamma:=\overline{\Omega^+} \cap \overline{\Omega^-} \subset \Omega,
\quad \mbox{with} \ \nabla a(x)\not= 0 \ \ \forall x\in \Gamma.
$$

This problem has been extensively studied in the literature (see  \cite{AL}, \cite{AT}, \cite{BCN}, \cite{BCN1},
\cite{DL}, \cite{Z} and the references therein). In order to prove the existence and multiplicity of positive
solutions, it is very important to obtain a priori estimates  on the
solutions. Blowing-up and re-scaling techniques of Gidas-Spruck \cite{GS} and Liouville
theorems are very useful in obtaining the a priori bound. Concerning
 problem (\ref{semi0}), the maxima of a sequence of solutions may
blow up on $\partial \Omega$, in $\Omega^+\cup \Omega^-$ or on $\Gamma$.
If the blow-up occurs on $\partial \Omega$ or in $\Omega^+\cup
\Omega^-$, we can use the classical Liouville theorems on
$\mathbb R^n_+$ or in $\mathbb R^n$ to derive a contradiction and hence obtain the a priori bound. If the blow-up
occurs on $\Gamma$ where $a(x)$ vanishes, then after blowing-up and re-scaling, one would arrive at the following limiting equation
\begin{equation}
\left \{
\begin{array}{lrl}
-\lap u= x_1u^p \quad \quad &\mbox{in} \ \mathbb R^n, \medskip \\
u\geq 0 \quad \quad &\mbox{in} \ \mathbb R^n. \label{inde}
\end{array}
\right.
\end{equation}

Berestycki, Capuzzo-Docetta, and Nirenberg \cite{BCN} proved that this equation has no positive solution
for
$1< p<\frac{n+2}{n-1}$ and thus
obtained a priori bound for problem \eqref{semi0} in this case. Here, the Liouville type theorem, the non-existence of solutions, was the key
to derive such a priori estimate.

Later Chen and Li \cite{CL},\cite{CL1} further relaxed the
restriction on $a(x)$ near $\Gamma$ and obtain a priori bound with a
general $p>1$.

Lin \cite{Lin} showed that the nonnegative
solution for
$$ -\lap u=x_1^mu^{n^\ast} \quad \quad \mbox{in} \ \mathbb R^n $$
is trivial, when $m$ is an odd positive integer and
$n^\ast=\frac{n+2}{n-2}$ is the critical exponent of Sobolev
imbedding.

Du and Li \cite{DL} considered nonnegative solution of the
problem
\begin{equation}
\left \{
\begin{array}{lll}
-\lap u=h(x_1) u^p \quad \quad &\mbox{in} \ \mathbb R^n, \medskip \\
\sup_{\mathbb R^n} u <\infty,
\end{array}
\right.
\end{equation}
where $h(t)=t|t|^{s}$ or $h(t)=(t^+)^{s}$ for some $s>0$ and $p>1$.
They showed that the solution is trivial.

Zhu \cite{Z1} investigated the indefinite nonlinear
boundary condition motivated by a prescribing sign-changing scalar
curvature problem on compact Riemannian manifolds with boundary.  He
proved that there exists no positive solution for
\begin{equation}
\left \{
\begin{array}{rll}
-\lap u=&0 \quad \quad &\mbox{in} \ \mathbb R^n_+,  \nonumber \medskip \\
\frac{\partial u}{\partial x_n}=&-x_1u^p \quad \quad &\mbox{on} \
\partial \mathbb R^n_+.
\end{array}
\right.
\end{equation}

\subsubsection{Nonlocal elliptic equations}
\,
\smallskip

If one considers the fractional indefinite problem
\begin{equation}
\left \{
\begin{array}{rll}
(-\lap)^s u= &a(x)u^p \quad \quad &\mbox{in} \ \Omega, \medskip \\
u=&0 \quad \quad &\mbox{on} \ \partial\Omega, \label{semi}
\end{array}
\right.
\end{equation}
and applies the blow-up technique,
one will also have to deal with the case that the blow-up occurs on $\Gamma$, and after re-scaling one will arrive at
following limiting equation
\begin{equation}\label{2.00}
(-\Delta)^{s} u(x) =  x_1u^p(x)\ \ \mbox{in}\ \ \R^n.
\end{equation}

To establish the non-existence of positive solutions, a commonly employed method involves the use of moving planes. This technique, initially introduced by Alexandroff and further developed by Berestycki and Nirenberg \cite{BN} et. al. was originally designed for local equations and cannot be applied directly to psuedo-differential equations involving the fractional Laplacian
due to its nonlocal nature. To overcome this challenge, Cafferelli and Silvestre \cite{CS} introduced an ``extension method'' capable of  transforming a non-local equation into a local one in higher dimensions.
Consequently, the traditional methods designed for local equations  can be applied to the extended problem to study the properties of solutions.
This innovative approach has led to a series of compelling results (refer to \cite{CLL,  CLZ, CW1, DLL, LW,LLW} and the references therein).

As an example, Chen and Zhu utilized the``extension method'' outline above  in \cite{CZhu}, and then applied the method of moving planes to show the monotonicity of  solutions for the extended problem, and hence derived the non-existence of positive bounded solutions for equation \eqref{2.00} in the case where $s\in[\frac{1}{2},1)$. This restriction on the value of $s$ was due to the nature of the extended equation.

Subsequently, Chen, Li and Zhu   \cite{CLZ} extended the range of $s$ from  $[\frac{1}{2},1)$ to $(0,1)$ by employing  a {\em direct method of moving planes} introduced by Chen, Li, and Li in \cite{CLL}. This method  significantly simplify the proof and has found widespread applications in establishing the symmetry, monotonicity, non-existence of solutions
for various elliptic equations and systems involving the fractional Laplacian, the fully nonlinear
nonlocal operators, the fractional p-Laplacians, and the higher order fractional operators. For a comprehensive review, please
refer to \cite{ CL,CL2,CLL,CLL1,CW,DQ,DQW,GMZ1,GMZ2,LW1,LW,LZ,LZ1,WuC} and the references therein.
\medskip

\subsubsection{Local and nonlocal parabolic equations}
\,
\smallskip

For indefinite local parabolic problems, Pol$\acute{a}\check{c}$ik and Quittner \cite{PQ} established the non-existence of bounded positive solutions of the following equation
\begin{equation} \label{integerL1}
\partial_t u(x,t)-\Delta u(x,t) =  a(x_1)f(u)\ \ \mbox{in}\ \ \R^n\tm\R.
\end{equation}

This kind of Liouville theorem plays an important role in deriving a priori estimates. It
can be employed to derive suitable a priori bounds for solutions of a family of corresponding equations through blowing-up and re-scaling arguments and to study the complete blow-up (see \cite{BV, MZ1, PQ, PQS, Q, QS}). It is well-known that these a priori estimates are important ingredients in obtaining the existence of
solutions of the same equations.

For indefinite fractional parabolic problems, Chen, Wu, and Wang \cite{CWW} modified the {\em direct method of moving planes} for nonlocal elliptic problems  such that it can be applied to indefinite fractional parabolic problems and thus established the non-existence of positive solutions for
\begin{equation}\label{2.02}
\partial_t u(x,t)+(-\Delta)^s u(x,t) =  x_1u^p(x,t)\ \ \mbox{in}\ \ \R^n\tm\R.
\end{equation}

\smallskip

So far, we have not seen the Liouville type theorems for solutions to nonlocal master equations \eqref{2.0} with indefinite nonlinearity, primarily owing to its full nonlocality and strongly correlated nature.  Traditional approaches, commonly applicable to the local heat operator $\partial_t-\Delta$
  as depicted in \eqref{integerL1} or to the fractional heat operator $\partial_t + (-\Delta)^s$ with local time derivative as shown in \eqref{2.02}, no longer suffice in this context.
  \medskip

\subsection{The non-locality of the master operator} \,

To illustrate the essential differences between the integer order parabolic operator and  the fractional order ones. let's consider
a simple version of maximum principle in the parabolic cylinder $\Omega\tm(0,1] $ with $\Omega$ being a bounded domain in $\R^n$.
\medskip

\begin{center}
\begin{tikzpicture}[scale=2]
 \draw[blue!30,fill=blue!30] (1,0) rectangle (2,1.5);
  \draw[blue!30,fill=blue!30] (-1,0) rectangle (-2,1.5);
  \draw[yellow!30,fill=yellow!30] (-1,0) rectangle (1,-1.5);
  \draw[green!30,fill=green!30] (1,0) rectangle (2,-1.5);
  \draw[green!30,fill=green!30] (-1,0) rectangle (-2,-1.5);
\draw [thick]  [black] [->, thick](-2.2,0)--(2.2,0) node [anchor=north west] {$x$};
\draw [thick]  [black!80][->, thick] (0,-1.7)--(0,1.7) node [black][ above] {$t$};
\path node at (-0.1,-0.1) {$0$};
\draw [thick] [dashed] [blue] (-2,1.5)--(2,1.5);
\draw [thick] [dashed] [blue] (1,1.5)--(1,-1.5);
\draw [thick] [dashed] [blue] (-1,1.5)--(-1,-1.5);

\draw[thick][red] (-1,0) -- (-1,1.5);
\draw[thick][red] (1,0) -- (1,1.5);
\draw [thick]  [green] (-1,0)--(1,0);
\draw[fill=red] (1,0) circle (0.02);
\draw[fill=red] (0,1.5) circle (0.02);
\draw[fill=red] (-1,0) circle (0.02);
\node at (0.1,1.62) {$1$};
\path  (0.06,0.55) [purple][semithick] node [ font=\fontsize{10}{10}\selectfont] {$\Omega\times(0,1]$};

\path  (0.06,0.85) [purple][semithick] node [ font=\fontsize{10}{10}\selectfont] {$\mathcal{L}u\geq 0$};
\path  (1.5,1) [purple][semithick] node [ font=\fontsize{10}{10}\selectfont] {$(I)$};
\path  (-1.5,1) [purple][semithick] node [ font=\fontsize{10}{10}\selectfont] {$(I)$};
\path  (-0.2,-0.5) [purple][semithick] node [ font=\fontsize{10}{10}\selectfont] {$(II)$};
\path  (1.5,-0.5) [purple][semithick] node [ font=\fontsize{10}{10}\selectfont] {$(III)$};
\path  (-1.5,-0.5) [purple][semithick] node [ font=\fontsize{10}{10}\selectfont] {$(III)$};
\node [below=0.5cm, align=flush center,text width=12cm] at  (0,-1.5)
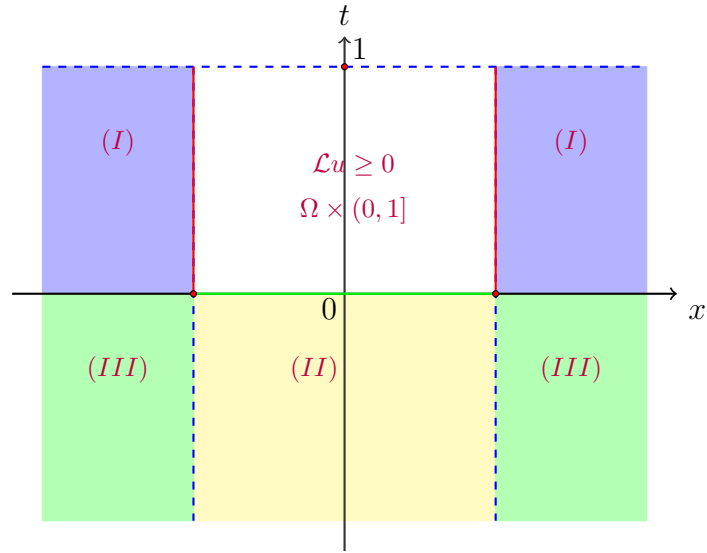
        {Figure 1. The region where the condition $u(x,t)\geq 0$ is required. };
\end{tikzpicture}
\end{center}

{\em (i) For the integer order heat operator:}

If \be\label{e1}\mathcal{L}u:=\partial_t u(x,t) -\Delta u(x,t) \geq 0, \;\;\; (x,t) \in \Omega \times (0,1],\ee
in order to derive that
 $$u(x,t) \geq 0 \mbox{ for } (x,t) \in \Omega \times (0,1], $$
one only needs to assume $u(x,t)\geq 0$  at the parabolic lateral boundary $\partial \Omega\tm(0,1]$ (the red line in Figure 1) and at the bottom boundary $\Omega\tm\{0\}$ (the green line in Figure 1).
\medskip

{\em (ii) For an integer order time derivative with the fractional Laplacian:}

If \be\label{e2}\mathcal{L}u:=\partial_t u(x,t) + (-\Delta)^s u(x,t) \geq 0, (x,t) \in \Omega \times (0,1],\ee
to ensure that $$u(x,t) \geq 0 \mbox{ for }  (x,t) \in \Omega \times (0,1],$$
one needs to require $u(x,t)\geq 0$   on  $\Omega^c\tm(0,1]$ (region $(I)$ in Figure 1) and also on the bottom boundary $\Omega\tm\{0\}$ (the green line in Figure 1).
\medskip

{\em (iii) For the dual fractional operator:}

If \be\label{e3}\mathcal{L}u:=\partial^\alpha_t u(x,t) + (-\Delta)^s u(x,t) \geq 0, (x,t) \in \Omega \times (0,1],\ee
to ensure that $$u(x,t) \geq 0 \mbox{ for }  (x,t) \in \Omega \times (0,1],$$
one needs to specify the initial condition for the entire past time before $t=0$ instead of only at the initial time moment $t=0$ and impose the boundary condition in the whole of $\Omega^c\tm(0,1]$ rather than on the parabolic lateral boundary $\partial \Omega\tm(0,1]$ (regions $(I)\ \mb{and} \ (II)$ in Figure 1).

{\em (iv) Now for the fully fractional master operator, the maximum principle reads as:

Assume that $u(x,t)$ is a solution of past-time and exterior values problem
\begin{equation}\label{IMP1}
\left\{
\begin{array}{ll}
    (\partial_t-\Delta)^su(x,t)\geq 0 ,~   &(x,t) \in  \Omega\times(0,1]  , \\
  u(x,t)\geq 0 , ~ &(x,t)  \in  (\mathbb{R}^n\setminus\Omega)\times(0,1),\\
  u(x,t)\geq 0 , ~ &(x,t)  \in \mathbb{R}^n\times(-\infty,0].
\end{array}
\right.
\end{equation}
Then $u(x,t)\geq 0$ in $\Omega\times(0,1]$} (see the detailed proof in Section \ref{5}).

Here due to the nonlocal and strongly correlated nature of the fully fractional master operator $(\partial_t-\Delta)^s$, in order to ensure the validity of the classical maximum principle, besides the exterior condition on
$(\mathbb{R}^n\setminus\Omega)\times(0,1)$ (region $(I)$ in Figure 1),
we must also require the past-time condition $u(x,t)\geq 0$ to hold on $\mathbb{R}^n\times(-\infty,0]$ (regions $(II)$ and $(III)$ in Figure 1), rather than just on $\Omega\times\{0\}$ or on $\Omega\times(-\infty,0]$ as required by the maximum principle for  equations \eqref{e2} and  \eqref{e3}, respectively.

There are counterexamples as provided in \cite{MGZ},
which reveal that if the  past-time condition is only satisfied on some part of $\mathbb{R}^n \times (-\infty,0]$, then the maximum principle is violated for the master operator.
\medskip

To circumvent such difficulties caused by the fully fractional heat operator, one needs to introduce new ideas and to develop  different new methods.

\subsection{Our main results and new ideas} \,

We are now ready to state our main results of the paper.
\medskip

\leftline{\bf (i) The case $p>1$.}
\smallskip

We first derive the monotonicity of solutions.
\smallskip

\begin{theorem}\label{thm1}
Let $$u(x,t)\in C^{2s+\epsilon,s+\epsilon}_{x,\, t,\,{\rm loc}}(\R^n\times\R)\cap \mathcal{L}(\mathbb{R}^n\times\mathbb{R})$$ be a  positive bounded classical solution of
\begin{equation}\label{2.111}
(\partial_t -\Delta)^{s} u(x,t) =  x_1u^p(x,t)\ \ \mbox{in}\ \ \R^n\times\R,
\end{equation}
where $0<s<1$ and $1<p<+\infty$.
Then for each $t\in\R,$ $u(\cdot,t)$ is  strictly monotone increasing  in $x_1$ direction.
\end{theorem}

Then based on the monotonicity result, we construct a sub-solution and use a contradiction argument to establish the non-existence of solutions.
\begin{theorem}\label{thm2}
 For any $1<p< +\infty$, equation  \eqref{2.111} possesses no positive bounded classical solutions.
\end{theorem}

\begin{remark}
To better illustrate the main ideas, we only consider the simple example as in equation \eqref{2.111}. Actually, the methods developed here are  also applicable to more general indefinite nonlinearity $a(x)f(u)$. Interested readers may work out the details.
\end{remark}
\bigskip

\leftline{\bf (ii) The case $p< 0$.}
\smallskip

In this case we can derive the non-existence of solutions without using their monotonicity.
\begin{theorem}\label{thm3}
 For any $p< 0$, the equation  \eqref{2.111} possesses no positive bounded classical solutions.
\end{theorem}

\leftline{\bf (iii) The case $0<p<1$.}
\smallskip

In this situation, we are still trying  to prove the monotonicity. Let's now assume that
\be\label{further-c}u(\cdot,t) \mb{ is   monotone increasing  in }x_1\mb{  direction for each } t\in\R.\ee
Then we can derive
\begin{theorem}\label{thm4}
For any $0<p< 1$, if $u$ satisfies \eqref{further-c},
then the  equation \eqref{2.111} possesses no bounded positive classical solutions.
\end{theorem}

The monotonicity of the positive solution $u(x,t)$ in $x_1$-direction will be proved by the method of moving planes. Let
$T_\lm = \{ x \in \R^n \mid x_1 = \lm \}$
be the plane perpendicular to $x_1$-axis, and $\Sigma_\lm$ be the region to the left of $T_\lm$. For each fixed $t$, we compare the value of $u(x,t)$ with $u(x^\lm, t)$, its value at the reflection point about $T_\lm$. Consider $$w_\lm (x,t) = u_\lm (x^\lm, t)-u(x,t).$$ Our main task is to prove that
$$ w_\lm (x,t) \geq 0, \;\; \mbox{ for all } (x,t) \in \Sigma_\lm \times \R.$$
This is usually done by a contradiction argument at a negative minimum of $w_\lm$. However, here under our assumption, $w_\lm$ is only bounded and a minimizing sequence of $w_\lm$ may leak to infinity. A traditional way to circumvent this challenge, whether dealing with integer-order  or fractional order elliptic and parabolic equations,  is to construct a specific auxiliary function (see  \cite{CLZ}\cite{CM}  \cite{CWW}\cite{CZhu} \cite{DL} \cite{Lin} \cite{PQ} and the references therein)
$$\bar{w}_\lm = \frac{w_\lm}{g}  \mbox{ with } g(x) \to \infty \mbox{ as } |x| \to \infty.$$
Such a function $\bar{w}_\lm$ shares the same sign with $w_\lm$ and decays to zero at infinity. Now if $w_\lm$ is negative somewhere, then
$\bar{w}_\lm$ attains its minimum at some point $x^o \in \Sigma_\lm$.

In this paper, we introduce entirely different new approaches:

(i) Instead of dividing $w_\lm$ by a function $g$, we subtract it by a sequence of cutoff functions, so that the new auxiliary functions are able to attain their minima. Through estimating the singular integrals defining
the fully fractional heat operator on the auxiliary functions, we will be able to derive a contradiction in the case if $w_\lm$ is negative somewhere in $\Sigma_\lm$. This new idea remarkably simplify the proof process, and
we  believe that it will find broad applications in studying other elliptic and parabolic fractional problems.

(ii) In the previous literature (for instance, see \cite{CLZ} \cite{CWW}), in the second step of the moving planes, the authors took limit of a sequence of equations to arrive at a
limiting equation and derived a contradiction at such a limiting equation. In order the sequence of equations to converge, some additional regularity assumptions on the higher order derivatives are required. In this paper, we adapt a new approach, estimating the singular integrals to derive a contradiction just along a sequence of equations without taking the limit. This new idea enable us to weaken the regularity assumption on the solutions, and better still, it can be applied to unbounded solutions.

(iii) In the proof of the nonexistence of solutions--Theorem \ref{thm2}, the non-separable nature of the master operator $(\partial_t - \Delta)^s$ also poses some challenge as explained below.

In \cite{CWW} and \cite{PQ}, to prove the nonexistence of bounded positive solutions, the authors consider
$$ \psi_R(t):=\jf_{\R^n}u(x,t) \varphi_R(x)\,\!dx=\jf_{B_2\lt(Re_1\rt)}u(x,t) \varphi_R(x)\,\!dx, $$
where $\varphi_R(x)$ is related to the first eigenfunction of fractional Laplacian $(-\Delta)^s$ on $B_1(Re_1)$ with $e_1 = (1, 0, \cdots, 0)$.

From the separable nature of the operator $\partial_t + (-\Delta)^s$,
one has
$$ \frac{d \psi_R (t)}{d t} =  \jf_{\R^n} \partial_t u(x,t) \varphi_R(x)\,\!dx = \int_{\mathbb{R}^n} \left[ -(-\Delta)^s u(x,t) +  x_1 u^p(x,t) \right] \varphi_R(x) dx. $$
Consequently, they showed that
for $R$ sufficiently large,
$$ \frac{d \psi_R (t)}{d t} \geq \psi_R(t).$$
 It follows that
$$ \psi_R(t) \to \infty \mbox{ as } t \to \infty.$$
This contradicts the boundedness of the solution $u(x,t)$ and hence proves the nonexistence.

For our master operator $(\partial_t - \Delta)^s$, the time and space derivative cannot be separated, hence obviously the aforementioned method is not applicable.
To overcome this difficulty, we develop a completely different approach. Instead of relying on $\psi_R(t)$, we compare $u(x,t)$ directly with the following function
\[v(x,t)=\phi_R(x)\eta(t),\mb{ with }\eta(t)=t^\beta-1,\]
 where $\phi_R(x)$ the first eigenfunction of fractional Laplacian $(-\Delta)^s$ on $B_1(Re_1)$ and $0<\beta=\frac{1}{2k+1}<s$ for some positive integer $k$. We prove that
 $v(x,t)$ is a sub-solution in $B_1(Re_1)\tm[1,T]$ for any $T>1$ and therefore derive a contraction with the bounded-ness of the solution $u(x,t)$. This new approach is much more general, it can also be applied to the situations in \cite{CWW} and \cite{PQ}.

\bigskip

This paper will be organized as follows.

In Section 2, we present the definitions of the local parabolic H\"{o}lder space $C^{2s+\epsilon,s+\epsilon}_{x,\, t,\, {\rm loc}}(\mathbb{R}^n\times\mathbb{R})$ 
 and also prove two maximum principles for the master operators.

In Section 3, we employ the method of moving planes to deduce the monotonicity of solutions and hence establish Theorem \ref{thm1}.

In Section 4, we construct a subsolution to derive the nonexistence of positive solutions and therefore prove Theorem \ref{thm2}, Theorem \ref{thm3}, and Theorem \ref{thm4}.


\section{Preliminaries}\label{5}
In this section, we collect definitions and derive auxiliary results that are needed in establishing our main theorems. Throughout this paper, $C$ will denote a positive constant whose value may vary from line to line.

We start by providing the definition of parabolic H\"{o}lder space
$$C^{2\alpha,\alpha}_{x,\, t}(\mathbb{R}^n\times\mathbb{R}),$$
which plays an essential role in ensuring  that the fully fractional heat operator $(\partial_t-\Delta)^s$ is well-defined (cf. \cite{Kry}). More precisely,
\begin{itemize}
\item[(i)]
For $0<\alpha\leq\frac{1}{2}$, we say that
$u(x,t)\in C^{2\alpha,\alpha}_{x,\, t}(\mathbb{R}^n\times\mathbb{R})$, if there exists a constant $C>0$ such that
\begin{equation*}
  |u(x,t)-u(y,\tau)|\leq C\left(|x-y|+|t-\tau|^{\frac{1}{2}}\right)^{2\alpha}
\end{equation*}
for any $x,\,y\in\mathbb{R}^n$ and $t,\,\tau\in \mathbb{R}$.
\item[(ii)]
For $\frac{1}{2}<\alpha\leq1$, we say that
$$u(x,t)\in C^{2\alpha,\alpha}_{x,\, t}(\mathbb{R}^n\times\mathbb{R}):=C^{1+(2\alpha-1),\alpha}_{x,\, t}(\mathbb{R}^n\times\mathbb{R}),$$ if $u$ is $\alpha$-H\"{o}lder continuous in $t$ uniformly with respect to $x$ and its gradient $\nabla_xu$ is $(2\alpha-1)$-H\"{o}lder continuous in $x$ uniformly with respect to $t$ and $(\alpha-\frac{1}{2})$-H\"{o}lder continuous in $t$ uniformly with respect to $x$.
\item[(iii)] While for $\alpha>1$, we say that
$u(x,t)\in C^{2\alpha,\alpha}_{x,\, t}(\mathbb{R}^n\times\mathbb{R}),$
if
$$\partial_tu,\, D^2_xu \in C^{2\alpha-2,\alpha-1}_{x,\, t}(\mathbb{R}^n\times\mathbb{R}).$$
\end{itemize}
In addition, we can analogously define the local parabolic H\"{o}lder space $C^{2\alpha,\alpha}_{x,\, t,\, \rm{loc}}(\mathbb{R}^n\times\mathbb{R})$.
\medskip

Next we establish  two (strong) maximum principle for the master operator $(\partial_t-\Delta)^s.$
\begin{theorem}\label{thm11}
Let $\Omega\subset\R^n$ be an open bounded set
and $[t_1, t_2]$ be an interval in $\R$. Suppose that $$w(x,t)\in C^{2s+\epsilon,s+\epsilon}_{x,\, t,\,{\rm loc}}(\R^n\times\R)\cap \mathcal{L}(\mathbb{R}^n\times\mathbb{R})$$ is a lower semi-continuous function on $\overline{\Omega}\times[t_1,t_2]$, satisfying
\begin{equation}\label{mp1000}
\begin{cases}
(\partial_t-\Delta)^s w(x,t)\geq 0,& \mb{ in } \Omega\tm(t_1,t_2], \\
w(x,t)
>0, &\mb{ in } (\R^n\tm(-\infty,t_2])\backslash (\Omega\tm(t_1,t_2]).
\end{cases}
\end{equation}
Then \be\label{sbp500}w(x,t)>0,\,\mb{for~all}\,(x,t)\in \Omega\tm(t_1,t_2].\ee
\end{theorem}
\begin{proof}[Proof]
If \eqref{sbp500} is not true, then there exists some $(x^0,t_0)\in\Omega\tm(t_1,t_2]$ such that
\[w(x^0,t_0)=\inf\limits_{\R^n\tm(-\infty,t_2]}w(x,t)<0.\]
Now by the definition of operator $(\partial_t-\Delta)^s$ and exterior condition in \eqref{mp1000} , we have
\[(\partial_t -\Delta)^s w(x^0, t_0)=C_{n,s}\jf_{-\infty}^{t_0}\jf_{\R^n}\frac{w (x^0,t_0)-w(y,\tau)}{(t_0-\tau)^{\frac{n}{2}+1+s}}e^{-\frac{|x^0-y|^2}{4(t_0-\tau)}}\operatorname{d}\!y\operatorname{d}\!\tau<0,\]
which contracts  differential inequality in \eqref{mp1000}.

Thus we verify \eqref{sbp500} and complete the proof of Theorem \ref{thm11}.
\end{proof}
\begin{theorem}\label{thm12}
Let $\Omega\subset\R^n$ be an open bounded set
and $[t_1, t_2]$ be an interval in $\R$. Suppose that   $$w(x,t)\in C^{2s+\epsilon,s+\epsilon}_{x,\, t,\,{\rm loc}}(\R^n\times\R)\cap \mathcal{L}(\mathbb{R}^n\times\mathbb{R})$$
is a lower semi-continuous function on $\overline{\Omega}\times[t_1,t_2]$, satisfying
\begin{equation}\label{sbp40}
\begin{cases}
(\partial_t-\Delta)^s w(x,t)\geq c(x,t) w(x,t),& \mb{ in } \Omega\tm(t_1,t_2], \\
w(x,t)
>0, &\mb{ in } (\R^n\tm(-\infty,t_2])\backslash (\Omega\tm(t_1,t_2]),
\end{cases}
\end{equation}
where $c(x,t)$ is continuous in $\Omega\tm(t_1,t_2]$. Then
\be\label{sbp50}w(x,t)>0,\,\mb{for~all}\,(x,t)\in \Omega\tm(t_1,t_2].\ee
\end{theorem}
\begin{proof}[Proof]
If \eqref{sbp50} is violated, then there must exist a first time $t_0\in(t_1,t_2]$ such that
\[w(x^0,t_0)=0, \mb{ for some } x^0\in\Omega,\]and
\[w(x,t)\geq0, \mb{ for all } x\in \R^n, t\in(-\infty,t_0].\]
Then by the definition of operator $(\partial_t-\Delta)^s$ and the exterior condition in \eqref{sbp40} , we have
\[(\partial_t -\Delta)^s w(x^0, t_0)=C_{n,s}\jf_{-\infty}^{t_0}\jf_{\R^n}\frac{w (x^0,t_0)-w(y,\tau)}{(t_0-\tau)^{\frac{n}{2}+1+s}}e^{-\frac{|x^0-y|^2}{4(t_0-\tau)}}\operatorname{d}\!y\operatorname{d}\!\tau<0,\]
which contracts  the differential inequality in \eqref{sbp40}, i.e. ,
\[(\partial_t -\Delta)^s w(x^0, t_0)\geq c(x^0,t_0) w(x^0,t_0)=0.\]
Now we arrive at  \eqref{sbp50} and   hence complete the proof of Theorem \ref{thm12}.
\end{proof}

\section{Monotonicity of solutions}
Consider \begin{equation}\label{2.1}
(\partial_t -\Delta)^{s} u(x,t) =  x_1u^p(x,t)\ \ \mbox{in}\ \ \R^n\times\R.
\end{equation}
We will use the direct method of moving planes to show that all positive solutions must be strictly monotone increasing along $x_1$ direction and thus prove Theorem \ref{thm1}.

 Before presenting   the main proof, we introduce the notation that will be used throughout the subsequent sections.

Let $$T_\lambda = \{(x_1, x') \in \mathbb{R}^n \mid x_1 = \lambda\},\, \lambda \in \mathbb{R}$$
be a moving planes perpendicular to the $x_1$-axis,
$$\Sigma_\lambda = \{x \in \mathbb{R}^n \mid x_1 < \lambda\}$$ be the region to the left of the hyperplane $T_\lambda$ in $\mathbb{R}^n$, and
$$x^\lambda = (2\lambda - x_1, x_2, \ldots, x_n)$$
be the reflection of $x$ with respect to the hyperplane $T_\lambda$.

Assume that  $u$ is a solution of pseudo differential equation \eqref{2.1}. To compare the values of $u(x,t)$ with $u_\lambda(x, t) = u(x^\lambda, t)$, we define $$w_\lambda(x, t) = u_\lambda(x, t) - u(x, t).$$ It is evident that $w_\lambda(x, t)$ is an antisymmetric function of $x$ with respect to the hyperplane $T_\lambda$.

It follows from equation \eqref {2.1} that
\be\label{maineq}
  \left.
\begin{array}{ll}
    (\partial_t-\Delta)^s w_\lm(x,t)&=x_1^\lm u_{\lm}^p(x,t)-x_1u^p(x,t)\\[0.2cm]
    &=[x_1^\lm-x_1]u^p_\lm+x_1[u^p_\lm-u^p]\\[0.2cm]
    &\geq px_1\xi^{p-1}_\lm(x,t) w_\lm(x,t),\end{array}
\right.
\ee
where $\xi_\lm$ lies between $u$ and $u_\lm$.

 We want to show that
\bee w_\lm (x,t) \geq 0, \;\forall (x,t)\in\Sigma_\lm\tm\R,\;\forall\lm\in\R.\eee

To this end, usually a contradiction argument is used. Suppose $w_\lm$ attains a negative minimum in $\Sigma_\lm$, then one would derive directly a contradiction with differential  inequality \eqref{maineq}. However, here, we only assume that $u$ is bounded without any decay condition  near infinity, which cannot  prevent the minimizing sequence of $w_\lm$ from leaking to infinity. To overcome this difficulty,  traditionally a common approach is to construct a specific auxiliary function
$$\bar{w}_\lm = \frac{w_\lm}{g}  \mbox{ with } g(x) \to \infty \mbox{ as } |x| \to \infty.$$

  In \cite{CLZ}, while addressing fractional elliptic equations, by exploiting a key estimate established on $w_\lambda$ at the negative minimum  $x^o$  of above $\bar{w}_\lm$:
\be \label{1003}
(-\Delta)^s w_\lambda (x^o) \leq \frac{C}{|(x^o)_1-\lambda|^{2s}}w_\lambda (x_o),
\ee the authors derived a contradiction.

While for fractional parabolic equations with integer order time derivative
$$ \partial_t u(x,t) + (-\lap)^s u(x,t) = x_1 u^p(x,t), \;\; (x,t) \in \mathbb{R}^n \times \mathbb{R},$$
the authors in \cite{CWW} modified the techniques in \cite{CLZ} so that they can be applied to fractional parabolic equations. They also relied on a similar estimate  as in  \eqref{1003}.

In this paper, we introduce a entirely different new approach. Instead of dividing $w_\lm$ by a function $g$, we subtract it by a sequence of cutoff functions, so that the new auxiliary functions are able to attain their minima. Through estimating the singular integrals defining
 the fully fractional heat operator on the auxiliary functions, we will be able to derive a contradiction in the case if $w_\lm$ is negative somewhere in $\Sigma_\lm$.

\begin{proof}[Proof of Theorem \ref{thm1}] \,

The proof will be accomplished in three steps.
\smallskip

In Step 1, we show that for $\lm\leq 0$, it holds
\be \label{mp} w_\lm(x,t) \geq 0, \;\forall (x,t)\in\Sigma_\lm\tm\R.\ee
This provides a starting point to move the plane.
\smallskip

In Step 2, we move the plane $T_\lm$ along $x_1$ direction as long as the above inequality holds.
We prove that this plane can be moved all the way to $\lm = \infty,$ that is, \eqref{mp} holds for all
real number $\lm$.
\smallskip

In Step 3, we further derive the strict inequality
$$  w_\lm(x,t) > 0, \;\forall (x,t)\in\Sigma_\lm\tm\R,\;\forall\lm\in\R,$$
which implies that for each fixed $t$, $u(x,t)$ is strictly increasing in $x_1$ direction. This result can be viewed as
a strong maximum principle.
\medskip

Now we carry out the details.
\medskip

Step 1. We argue by  contradiction. If \eqref{mp} is not true, since $u$ is bounded, then there exist   some $\lambda\leq 0$  and  a constant $A>0$ such that
\begin{equation}\label{2.3}
\inf\limits_{(x,t)\in \Sigma_\lm\times \R}w_\lambda(x,t):=-A<0.
\end{equation}

If a minimum of $w_\lambda$ is attained at some point $(x^o, t_o)$, then we can derive a contradiction immediately at that point. Unfortunately, this is usually not the case, because both $x$ and $t$ are in unbounded domains.

Nonetheless, there exist a sequence of approximate minima $\{(x^k,t_k)\}\subset \Sigma_\lm\times\R$ and a  sequence $\{\varepsilon_k\}\searrow 0$ such that
\begin{equation}\label{2.4}
w_\lm(x^k,t_k)=-A+\varepsilon_k<0.
\end{equation}

To obtain a sequence of functions that attain their minima, we make a perturbation  of  $w_\lm$ near $(x^k,t_k)$ as the following
\begin{equation}\label{2.5}
v_k(x,t)=w_\lm(x,t)-\varepsilon_k\eta_k(x,t) \ \mbox{in}\ \R^n\times \R,
\end{equation}
where
\[\eta_k(x,t)=\eta\lt(\fr{x-x^k}{r_k},\fr{t-t_k}{r_k^2}\rt),\]
with $r_k=\frac{1}{2}{\rm {dist}}(x^k,T_{\lm})>0$ and 
  $\eta\in C^\infty_0(\R^n\times\R)$ is a cut-off smooth function satisfying
\begin{equation*}
\left\{\begin{array}{r@{\ \ }c@{\ \ }ll}
0\leq \eta\leq 1 &\mbox{in}&\ \ \R^n\times\R\,, \\[0.05cm]
 \eta= 1 &\mbox{in}&\ \ B_{1/2}(0)\times[-\fr{1}{2},\fr{1}{2}]\,, \\[0.05cm]
\eta= 0 &\mbox{in}&\ \ \left(\R^n\times(-\infty,1]\right) \backslash\left(B_{1}(0)\times[-1,1]\right)\,. \\[0.05cm]
\end{array}\right.
\end{equation*}
Denote
\[Q_k(x^k,t_k):=B_{r_k}(x^k)\times\lt[t_k-r_k^{2}, t_k+r_k^{2}\rt]\subset\Sigma_\lm\tm\R,\]
the parabolic cylinder centered at $(x^k, t_k)$.

 By \eqref{2.3}, \eqref{2.4} and \eqref{2.5}, we have
\begin{equation*}
\left.\begin{array}{r@{\ \ }c@{\ \ }ll}
v_k(x^k,t_k)&=&-A\,, \\[0.15cm]
v_k(x,t)=w_\lambda(x,t)&\geq& -A\ \mbox{in}\ \left(\Sigma_\lm\times\R\right) \backslash Q_k(x^k,t_k)\,, \\[0.15cm]
v_k(x,t)=-\varepsilon_k\eta_k(x,t) &>& -A\ \mbox{on}\ \ T_\lm\tm\R\,. \\[0.05cm]
\end{array}\right.
\end{equation*}
This implies that each $v_k$ must attain its minimum which is at most $-A$ at some point $(\bar{{x}}^k,\bar{t}_k)$ in the parabolic cylinder $\overline{Q_k(x^k,t_k)}\subset{\Sigma}_\lm\times \R$, that is, 
\begin{equation}\label{2.6}
\exists\ \{(\bar{{x}}^k,\bar{t}_k)\}\subset\overline{Q_k(x^k,t_k)}\ \   \ s.t.\ \  \   -A-\varepsilon_k\leq v_k(\bar{{x}}^k,\bar{t}_k)=\inf\limits_{\Sigma_\lm\times\R}v_k(x,t)\leq -A.
\end{equation}
Here we have used  \eqref{2.3} and \eqref{2.5}.
It follows that
\begin{equation}\label{mp1}
\left.\begin{array}{r@{\ \ }c@{\ \ }ll}
-A&\leq &w_\lambda(\bar{x}^k,\bar{t}_k)\leq-A+\varepsilon_k<0\,.
\end{array}\right.
\end{equation}

In addition, starting from the definition of operator $(-\Delta)^s$ and utilizing the antisymmetry of
$w_\lambda$ in $x$ as well as  the fact that
$$\abs{\bar{x}^k-y^\lm}>\abs{\bar{x}^k-y}\mb{ for }y\in \Sigma_\lambda,$$ and by \eqref{2.6}, we arrive at
\begin{equation}\label{2.8}
\left.\begin{array}{r@{\ \ }c@{\ \ }ll}
(\partial_t-\Delta)^s v_k(\bar{x}^k,\bar{t}_k)&=&C_{n,s}\jf_{-\infty}^{\bar{t}_k}\jf_{\R^{n}}\frac{v_k(\bar{x}^k,\bar{t}_k)-v_k(y,\tau)}{(\bar{t}_k-\tau)^{\frac{n}{2}+1+s}}e^{-\frac{|\bar{x}^k-y|^2}{4(\bar{t}_k-\tau)}}\operatorname{d}\!y\operatorname{d}\!\tau \\[0.3cm]
&=&C_{n,s}\jf_{-\infty}^{\bar{t}_k}\jf_{\Sigma_\lm}\frac{v_k(\bar{x}^k,\bar{t}_k)-v_k(y,\tau)}{(\bar{t}_k-\tau)^{\frac{n}{2}+1+s}}e^{-\frac{|\bar{x}^k-y|^2}{4(\bar{t}_k-\tau)}}\operatorname{d}\!y\operatorname{d}\!\tau \\[0.3cm]&\,&\,+C_{n,s}\jf_{-\infty}^{\bar{t}_k}\jf_{\Sigma_\lm}\frac{v_k(\bar{x}^k,\bar{t}_k)-v_k(y^\lm,\tau)}{(\bar{t}_k-\tau)^{\frac{n}{2}+1+s}}e^{-\frac{|\bar{x}^k-y^\lm|^2}{4(\bar{t}_k-\tau)}}\operatorname{d}\!y\operatorname{d}\!\tau \\[0.3cm]
&\leq&C_{n,s}\jf_{-\infty}^{\bar{t}_k}\jf_{\Sigma_\lm}\frac{2v_k(\bar{x}^k,\bar{t}_k)-v_k(y,\tau)-v_k(y^\lambda,\tau)}{(\bar{t}_k-\tau)^{\frac{n}{2}+1+s}}e^{-\frac{|\bar{x}^k-y^\lm|^2}{4(\bar{t}_k-\tau)}}\operatorname{d}\!y\operatorname{d}\!\tau \\[0.3cm]
&\leq&C_{n,s}2\lt(v_k(\bar{x}^k,\bar{t}_k)+\ve_k\rt)\jf_{-\infty}^{\bar{t}_k}\jf_{\Sigma_\lambda}\frac{1}{(\bar{t}_k-\tau)^{\frac{n}{2}+1+s}}e^{-\frac{|\bar{x}^k-y^\lm|^2}{4(\bar{t}_k-\tau)}}\operatorname{d}\!y\operatorname{d}\!\tau
\\
&\leq&\fr{C_1(-A+\ve_k)}{r_k^{2s}}.\end{array}\right.
\end{equation}
Hence, a combination of \eqref{2.5} and  \eqref{2.8} yields  that
\begin{eqnarray*}\nonumber
(\partial_t-\Delta)^sw_\lm(\bar{x}^k,\bar{t}_k)&=&(\partial_t-\Delta)^s v_k(\bar{x}^k,\bar{t}_k)+\varepsilon_k(\partial_t-\Delta)^s \eta_k(\bar{x}^k,\bar{t}_k)
 \\  \nonumber
&\leq&\fr{C_1(-A+\ve_k)}{r_k^{2s}}+\fr{C_2\ve_k}{r_k^{2s}}
\\
&\leq& \fr{C(-A+\ve_k)}{r_k^{2s}},
\end{eqnarray*}
where we have used the following  estimates
\[(\partial_t-\Delta)^s \eta_k(\bar{x}^k,\bar{t}_k)\leq \fr{C_2}{r_k^{2s}}.\]
One can  find  the detailed proof  in  corollary2.2 in \cite{CM1}.

Then, together with differential inequality \eqref{maineq} and \eqref{mp1} as well as the fact that $\lm\leq 0,$ we derive
\be\label{mp2}\fr{C(-A+\ve_k)}{r_k^{2s}}\geq p\bar{x}_1^k\xi^{p-1}_\lm(\bar{x}^k,\bar{t}_k) w_\lm(\bar{x}^k,\bar{t}_k)\geq 0.\ee
 This directly implies a contradiction and hence verifies \eqref{mp}.
\medskip

Step 2. Inequality \eqref{mp} provides a starting point to move the plane.
Now we move plane $T_\lm$ towards  the right along $x_1-$direction as long as the inequality \eqref{mp} holds to its limiting position $T_{\lm_0}$ with $\lm_0$ defined by
\[\lm_0:=\sup\{\lm\geq 0:w_\mu\geq0\,\,\text{in}\,\Sigma_\mu\times\R,\mu\leq\lm\}.\]

 We will show that $$\lm_0=+\infty.$$
Otherwise, if $\lm_0<+\infty,$ then by its definition, there exist a sequence $\{\lm_k\}\searrow\lm_0$ and a sequence of positive numbers $\{m_k\}$ such that
\be
\inf\limits_{(x,t)\in \Sigma_{\lm_k}\times \R}w_{\lambda_k}(x,t):=-m_k<0.
\ee

First, we show that \be \label{mf4}m_k\to 0,\,\mb{as}\,\,k \to \infty.\ee
If not, there is a subsequence, still denoted  by $\{m_k\},$ such that $$-m_k<-M \mb{ for some } M>0.$$ Consequently, there exists a sequence $\{y^k,s_k\}\subset\Sigma_{\lambda_k}\tm\R$ such that \be\label{mp3}w_{\lambda_k}(y^k,s_k)\leq -M<0.\ee
Case 1. If $y^k\in \Sigma_{\lm_k}\backslash\Sigma_{\lm_0}$, then by virtue of $\lm_k\to\lm_0,$ we have $$\abs{y^k-(y^k)^{\lm_k}}=2\abs{\lm_k-y^k_1}\to 0, \,\mb{as}\,k\to\infty.$$
The regularity result in \cite{ST} implies that  $u$ is   uniformly continuous,  and hence
\[w_{\lambda_k}(y^k,s_k)=u((y^k)^{\lm_k},s_k)-u(y^k,s_k)\to 0, \,\mb{as}\,\,k\to \infty.\]
Case 2. If $y^k\in \Sigma_{\lm_0}$, then combining the fact that  $\lm_k\to\lm_0,$  the   uniform continuity of $u$, and the definition of $\lm_0$, we deduce
\begin{eqnarray*}\nonumber w_{\lambda_k}(y^k,s_k)&=&u((y^k)^{\lm_k},s_k)-u((y^k)^{\lm_0},s_k)+w_{\lm_0}(y^k,s_k)\\
&\geq& u((y^k)^{\lm_k},s_k)-u((y^k)^{\lm_0},s_k)\to 0, \,\mb{as}\,\,k\to \infty.
\end{eqnarray*}
In the above two possible cases, we all derive contradictions with \eqref{mp3} and thus verify \eqref{mf4}.

Now from \eqref{mf4}, there exists a sequence $\{(x^k,t_k)\}\subset \Sigma_{\lm_k}\tm\R$ such that
\[w_{\lambda_k}(x^k,t_k)=-m_k+m_k^2<0.\]
Making a perturbation of $w_{\lm_k}$ near $(x^k,t_k)$ as follows
\begin{equation*}
v_k(x,t)=w_{\lm_k}(x,t)-m_k^2\eta_k(x,t), \;\; \; (x,t) \in \R^n\times \R,
\end{equation*}
where $\eta_k$ is defined as in Step 1 with $r_k=\frac{1}{2}{\rm dist} \{x^k,T_{\lm_k}\}.$\\

Denote
\[P_k(x^k,t_k):=B_{r_k}(x^k)\times\lt[t_k-r_k^{2}, \, t_k+r_k^{2}\rt]\subset\Sigma_{\lm_k}\tm\R.\]
Each $v_k$ must attain its minimum which is at most $-m_k$ in $\overline{P_k(x^k,t_k)}\subset{\Sigma}_{\lm_k}\times \R$, that is, 
\begin{equation*}
\exists\ \{(\bar{{x}}^k,\bar{t}_k)\}\subset\overline{P_k(x^k,t_k)}\ \   \ s.t.\ \  \   -m_k-m_k^2\leq v_k(\bar{{x}}^k,\bar{t}_k)=\inf\limits_{\Sigma_{\lm_k}\times\R}v_k(x,t)\leq -m_k.
\end{equation*}
This implies that
\begin{equation}\label{mf5}
-m_k\leq w_{\lm_k}(\bar{{x}}^k,\bar{t}_k)\leq -m_k+m_k^2<0.
\end{equation}

Using a similar argument as in Step 1, we derive
\be\label{mp6}\fr{C(-m_k+m_k^2)}{r_k^{2s}}\geq p\bar{x}_1^k\xi^{p-1}_{\lm_k}(\bar{x}^k,\bar{t}_k) w_{\lm_k}(\bar{x}^k,\bar{t}_k).\ee
If  $\bar{x}_1^k\leq0,$ then we have done by virtue of the proof in Step 1.
 Now we  assume that $0<\bar{x}_1^k\leq \lm+1$ for sufficiently large $k$.

 Then by \eqref{mf5} and \eqref{mp6}, we further obtain
\be\label{mp7}{C(1-m_k)}\leq p\bar{x}_1^kr_k^{2s}\xi^{p-1}_{\lm_k}(\bar{x}^k,\bar{t}_k).\ee
 Taking into account that $u$ is bounded and $p>1$, we arrive at \be\label{mp8}\bar{x}_1^k, \; r_k, \; u(\bar{x}^k,\bar{t}_k)\geq c>0,\ee for sufficiently large $k$.

Owing to  $w_\lm(\bar{x}^k,\bar{t}_k)\to 0 \,\mb{as}\,\, k\to\infty,$  it follows  that  \be\label{mp10}u_{\lm_k}(\bar{x}^k,\bar{t}_k)\geq c>0,\ee for sufficiently large $k.$

More accurately, from the initial equation \eqref{maineq}, we are able to modify \eqref{mp6} as
\begin{eqnarray}\label{mp9}& &\fr{C(-m_k+m_k^2)}{r_k^{2s}}\geq 2\lt[{\lm_k}-\bar{x}_1^k\rt]u^p_{\lm_k}(\bar{x}^k,\bar{t}_k)+ p\bar{x}_1^k\xi^{p-1}_{\lm_k}(\bar{x}^k,\bar{t}_k) w_{\lm_k}(\bar{x}^k,\bar{t}_k).\end{eqnarray}

Now
a combination of \eqref{mf5}, \eqref{mp8}, \eqref{mp9}, \eqref{mp10} and the fact that  $|\lm_k-\bar{x}^k_1|\sim r_k$, yields a contradiction, and hence we must have $\lm_0=+\infty.$ This completes Step 2.
\medskip

Step 3. In the above step, we have shown that
$$  w_\lm(x,t) \geq  0, \;\forall (x,t)\in\Sigma_\lm\tm\R,\;\forall\lm\in\R.$$
Now, we will further prove that the strict inequality holds:
\be \label{mp>}  w_\lm(x,t) >  0, \;\forall (x,t)\in\Sigma_\lm\tm\R,\;\forall\lm\in\R. \ee

Otherwise, for some fixed $\lm$ there exists a point $(x^o, t_o)$ in $\Sigma_\lm \times \mathbb{R}$, such that
$$ w_\lm (x^o, t_o) = \min_{\Sigma_\lm \times \mathbb{R}} w_\lm (x,t) = 0.$$
Then from differential inequality  \eqref{maineq}, we derive
$$(\partial_t-\lap)^s w_\lm (x^o, t_o) \geq 0. $$
On the other hand, at a minimum point $(x^o, t_o)$ in $\Sigma_\lm \times \mathbb{R}$, we must have

\begin{eqnarray}\label{1010}
 (\partial_t-\Delta)^s w_\lm (x^o, t_o)&=&C_{n,s}\jf_{-\infty}^{t_o}\jf_{\R^{n}}\frac{w_\lm (x^o, t_o)-w_\lm(y,\tau)}{(t_o-\tau)^{\frac{n}{2}+1+s}}e^{-\frac{|x^o-y|^2}{4(t_o-\tau)}}\operatorname{d}\!y\operatorname{d}\!\tau\nonumber\\
&=&C_{n,s}\jf_{-\infty}^{t_o}\jf_{\Sigma_\lm}\frac{w_\lm (x^o, t_o)-w_\lm(y,\tau)}{(t_o-\tau)^{\frac{n}{2}+1+s}}e^{-\frac{|x^o-y|^2}{4(t_o-\tau)}}\operatorname{d}\!y\operatorname{d}\!\tau
\nonumber\\&\quad&+C_{n,s}\jf_{-\infty}^{t_o}\jf_{\Sigma_\lm}\frac{w_\lm (x^o, t_o)-w_\lm(y^\lm,\tau)}{(t_o-\tau)^{\frac{n}{2}+1+s}}e^{-\frac{|x^o-y^\lm|^2}{4(t_o-\tau)}}\operatorname{d}\!y\operatorname{d}\!\tau
\nonumber\\&=&C_{n,s}\jf_{-\infty}^{t_o}\jf_{\Sigma_\lm}\fr{w_\lm (x^o, t_o)-w_\lm(y,\tau)}{(t_o-\tau)^{\frac{n}{2}+1+s}}[e^{-\frac{|x^o-y|^2}{4(t_o-\tau)}}-e^{-\frac{|x^o-y^\lm|^2}{4(t_o-\tau)}}]\operatorname{d}\!y\operatorname{d}\!\tau
\nonumber\\&\quad&+2C_{n,s}w_\lm (x^o, t_o)\jf_{-\infty}^{t_o}\jf_{\Sigma_\lm}\frac{1}{(t_o-\tau)^{\frac{n}{2}+1+s}}e^{-\frac{|x^o-y^\lm|^2}{4(t_o-\tau)}}\operatorname{d}\!y\operatorname{d}\!\tau\nonumber\\
&=&C_{n,s}\jf_{-\infty}^{t_o}\jf_{\Sigma_\lm}\fr{w_\lm (x^o, t_o)-w_\lm(y,\tau)}{(t_o-\tau)^{\frac{n}{2}+1+s}}[e^{-\frac{|x^o-y|^2}{4(t_o-\tau)}}-e^{-\frac{|x^o-y^\lm|^2}{4(t_o-\tau)}}]\operatorname{d}\!y\operatorname{d}\!\tau
\nonumber\\&\leq& 0.\nonumber\end{eqnarray}
Consequently,
\begin{eqnarray*}C_{n,s}\jf_{-\infty}^{t_o}\jf_{\Sigma_\lm}\fr{w_\lm (x^o, t_o)-w_\lm(y,\tau)}{(t_o-\tau)^{\frac{n}{2}+1+s}}[e^{-\frac{|x^o-y|^2}{4(t_o-\tau)}}-e^{-\frac{|x^o-y^\lm|^2}{4(t_o-\tau)}}]\operatorname{d}\!y\operatorname{d}\!\tau=0,\end{eqnarray*}
which implies that
$$ w_\lm (x, t) \equiv 0, \;\; \forall \, (x, t) \in \Sigma_\lm \times (-\infty, t_o].$$
Therefore,
$$ u_\lm(x,t) \equiv u(x,t), \; \forall \, (x, t) \in \mathbb{R}^n \times (-\infty, t_o].$$
This contradicts the second equality in \eqref{maineq}:
$$ (\partial_t -\Delta)^s w_\lm(x,t) = [x_1^\lm-x_1]u_\lm^p+x_1[{u_\lm^p-u^p}],
$$
since  the right hand side of the above
is greater than zero while the left hand side is zero.

Therefore \eqref{mp>} must be valid and $u(x,t)$ is strictly monotone increasing along the $x_1$-direction. This completes the proof of Theorem\ref{thm1}.
\end{proof}
\medskip

\section{Non-existence of solutions}
In the previous  section, we  have shown that each positive solution $u(x,t)$ of
\begin{equation}\label{3.1}
(\partial_t-\Delta)^{s} u(x,t) =  x_1u^p(x,t)\ \ \mbox{in}\ \ \R^n\times\R,
\end{equation} is monotone increasing along $x_1$-direction. Based on this, in this section, we will derive a contradiction   to establish the non-existence of positive solutions to  \eqref{3.1} first in the case of $1<p<+\infty$ and hence prove Theorem \ref{thm2}. Then  by applying a different approach, we extend this non-existence result to   the other two cases: $p<0$ and $0<p<1.$
\medskip

\begin{proof}[Proof of Theorem \ref{thm2}]

Let $\lm_1\mb{ and }\phi$ be the first eigenvalue and eigenfunction of the problem
 \begin{equation}
\begin{cases}
(-\Delta)^s\phi(x)=\lm_1\phi(x), &x\in B_1(0), \\
\phi(x)=0, &x\in B^c_1(0).
\end{cases}
\end{equation}

We may assume that
\be\label{mp12}\max\limits_{\R^n}\phi(x)=1.\ee
For any $R\geq 1$, denote \be\label{phi}\phi_R(x)=\phi(x-Re_1),\ee
where $e_1$ is a unit vector in $x_1$-direction.

Let \[v(x,t)=\phi_R(x)\eta(t)\mb{ with }\eta(t)=t^\beta-1,\]
 where $0<\beta=\frac{1}{2k+1}<s$ for some positive integer $k$.
\begin{center}
\begin{tikzpicture}[scale=2]
 \draw[blue!30,fill=blue!30] (1,0) rectangle (2,1);
  \draw[blue!30,fill=blue!30] (-1,0) rectangle (-2,1);
  \draw[yellow!30,fill=yellow!30] (-1,0) rectangle (1,-1.5);
  \draw[green!30,fill=blue!30] (1,1) rectangle (2,-1.5);
  \draw[green!30,fill=blue!30] (-1,1) rectangle (-2,-1.5);
\draw [thick]  [black] [->, thick](-2.2,-0.5)--(2.2,-0.5) node [anchor=north west] {$x$};
\draw [thick]  [black!80][->, thick] (0,-1.5)--(0,1.5) node [black][ above] {$t$};
\path [blue]node at (-0.1,0.4) {$B_1(Re_1)\tm[1,T]$};

\draw [thick] [dashed] [blue] (-2,1)--(2,1);
\draw [thick] [dashed] [blue] (1,1)--(1,-1.5);
\draw [thick] [dashed] [blue] (-1,1)--(-1,-1.5);

\draw [thick] [red] (-1,0)--(1,0);
\draw [thick]  [red] (1,0)--(1,1);
\draw [thick]  [red] (-1,0)--(-1,1);
\draw[fill=red] (0,0) circle (0.02);
\draw[fill=red] (0,1) circle (0.02);
\node at (0.1,-0.1) {$1$};
\node at (0.1,1.12) {$T$};
\path  (1.5,0.5) [purple][semithick] node [ font=\fontsize{10}{10}\selectfont] {$v(x,t)=0$};
\path  (-1.5,0.5) [purple][semithick] node [ font=\fontsize{10}{10}\selectfont] {$v(x,t)= 0$};
\path  (-0.1,-0.8) [purple][semithick] node [ font=\fontsize{10}{10}\selectfont] {$v(x,t)\leq 0$};
\path  (1.5,-0.8) [purple][semithick] node [ font=\fontsize{10}{10}\selectfont] {$v(x,t)=0$};
\path  (-1.5,-0.8) [purple][semithick] node [ font=\fontsize{10}{10}\selectfont] {$v(x,t)=0$};
\node [below=0.5cm, align=flush center,text width=12cm] at  (0,-1.5)
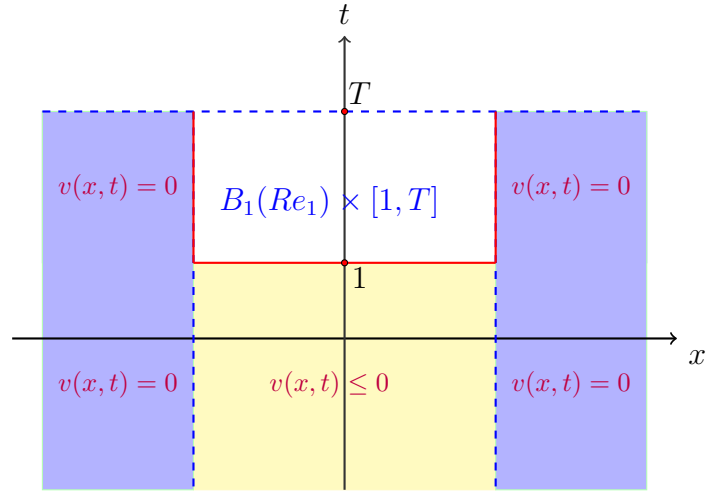
        {Figure 2. Distribution of values of function $v(x,t)$. };
\end{tikzpicture}
\end{center}

 We use a contradiction argument. Assume that $u$ is a positive bounded solution of \eqref{3.1}.
 Denote \be\label{mp11}T=\lt({M}+1\rt)^{1/\beta}\,\,\; \mb{ with}\,\,\; M:=\sup\limits_{\R^n\tm\R}u(x,t). \ee

Then by a direct calculation, we derive that for all $(x,t)\in B_1(Re_1)\tm[1,T],$
\begin{eqnarray}\label{mp15}
(\partial_t-\Delta)^sv(x,t)&=&C_{n,s}\jf_{-\infty}^{t}\jf_{\R^n}\frac{\phi_R (x)\eta(t)-\phi_R (y)\eta(\tau)}{(t-\tau)^{\frac{n}{2}+1+s}}e^{-\frac{|x-y|^2}{4(t-\tau)}}\operatorname{d}\!y\operatorname{d}\!\tau
\nonumber \\&=&C_{n,s}\jf_{-\infty}^{t}\jf_{\R^n}\frac{(\phi_R (x)-\phi_R (y))\eta(t)+(\eta(t)-\eta(\tau))\phi_R (y)}{(t-\tau)^{\frac{n}{2}+1+s}}e^{-\frac{|x-y|^2}{4(t-\tau)}}\operatorname{d}\!y\operatorname{d}\!\tau
\nonumber\\&\leq&\eta(t)(-\Delta)^s\phi_R(x)+ \sup\limits_{\R^n}\phi_R (x)\partial_t^s\eta(t)\nonumber\\&\leq&(\lm_1\eta(t)+C_st^{\gamma-\beta})\sup\limits_{\R^n}\phi_R(x)\nonumber\\&\leq& \lm_1(T^\beta-1)+C_s\nonumber\\&:=& C_{T}.\end{eqnarray}

 On the other hand, it follows from equation \eqref{3.1} that for all $(x,t)\in B_1(Re_1)\tm[1,T],$
 \begin{eqnarray}\label{mp16}
 (\partial_t-\Delta)^s u(x,t)&\geq& (R-1)u^p(x,t)\nonumber\\
 &\geq& (R-1)m_T^p,
 \end{eqnarray}
 with
\[m_T:=\inf\limits_{B_1(0)\tm[1,T]}u(x,t)>0.\]
 Here we have used the monotonicity of $u$ in $x_1$-direction derived  from Theorem \ref{thm1}.

  Now we choose sufficiently large  $R$ such that \be\label{mp13}(R-1)m_T^p>{C}_T.\ee

Let $$w(x,t)=u(x,t)-v(x,t),$$then by \eqref{mp15}, \eqref{mp16},\eqref{mp13}, $w$ satisfies
\begin{equation}\label{mp17}
\begin{cases}
(\partial_t-\Delta)^s w(x,t)\geq 0,& \mb{ in } B_1(Re_1)\tm(1,T], \\
w(x,t)
>0, &\mb{ in } (\R^n\tm(-\infty,T])\backslash (B_1(Re_1)\tm(1,T]).
\end{cases}
\end{equation}

Consequently, by the  maximum principle (Theorem \ref{thm11}), we conclude

\be\label{mp18}w(x,t)>0,\,\mb{for~all}\,(x,t)\in B_1(Re_1)\tm(1,T].\ee
 That is,
\[u(x,t)>\phi_R(x)(t^\beta-1) \,\,\mb{in}\,\,B_1(Re_1)\tm(1,T],\]
and then
\[M>\max\limits_{\R^n}\phi_R(x)(T^\beta-1).\]
This  is a contradiction and  implies that equation  \eqref{3.1} possesses no positive bounded solution and  thus completes the proof of Theorem\ref{thm2}.
\end{proof}

\begin{proof}[Proof of Theorem \ref{thm3}]

Let $v(x,t)$ and $C_T$ be defined as in the proof of Theorem  \ref{thm2}, we will use a similar argument  to derive a contradiction.

From the proof of Theorem \ref{thm2}, we have
\begin{eqnarray}\label{sup}
(\partial_t-\Delta)^sv(x,t)\leq C_{T}, \, (x,t)\in B_1(Re_1)\tm[1,T].\end{eqnarray}

  Here \be\label{sup1}T=\lt({M}+1\rt)^{1/\beta}\,\,\; \mb{ with}\,\,\; M:=\sup\limits_{\R^n\tm\R}u(x,t). \ee

 Assume that $u$ is a positive bounded solution of \eqref{3.1},
  then since   $p<0$,  we have that  for all $(x,t)\in B_1(Re_1)\tm[1,T],$
 \begin{eqnarray}\label{sup2}
 (\partial_t-\Delta)^s u(x,t)&\geq& (R-1)u^p(x,t)\nonumber\\
 &\geq& (R-1)M^p.
 \end{eqnarray}

Comparing with the proof of Theorem \ref{thm2}, instead of $m_T$ in \eqref{mp13}, we use $M$ in \eqref{sup2}, since $p<0,$ we do not need the monotonicity of $u$ in $x_1$-direction.

  Now we choose sufficiently large  $R$ such that \be\label{sup3}(R-1)M^p>{C}_T.\ee

Then similar to the process of proof of Theorem \ref{thm2}, one can deduce   a contradiction with \eqref{sup1}, which  implies that equation  \eqref{3.1} possesses no positive bounded solution and  thus completes the proof of Theorem\ref{thm3}.\end{proof}

\begin{proof}[Proof of Theorem \ref{thm4}]

 By using a similar argument as in the proof of Theorem \ref{thm2},
we are able to modify \eqref{mp15} as

\begin{eqnarray}\label{sbp1}
(\partial_t-\Delta)^sv(x,t)&=&C_{n,s}\jf_{-\infty}^{t}\jf_{\R^n}\frac{\phi_R (x)\eta(t)-\phi_R (y)\eta(\tau)}{(t-\tau)^{\frac{n}{2}+1+s}}e^{-\frac{|x-y|^2}{4(t-\tau)}}\operatorname{d}\!y\operatorname{d}\!\tau
\nonumber \\&=&C_{n,s}\jf_{-\infty}^{t}\jf_{\R^n}\frac{(\phi_R (x)-\phi_R (y))\eta(t)+(\eta(t)-\eta(\tau))\phi_R (y)}{(t-\tau)^{\frac{n}{2}+1+s}}e^{-\frac{|x-y|^2}{4(t-\tau)}}\operatorname{d}\!y\operatorname{d}\!\tau
\nonumber\\&\leq&\eta(t)(-\Delta)^s\phi_R(x)+ \sup\limits_{\R^n}\phi_R (x)\partial_t^s\eta(t)\nonumber\\&\leq&\lm_1v(x,t)+C_s\fr{u(x,t)}{m_T},\end{eqnarray}
for each $(x,t)\in B_1(Re_1)\tm[ 1,T]$, where
\[m_T:=\inf\limits_{B_1(0)\tm[1,T]}u(x,t)>0.\]
In the above  we have used assumption \eqref{further-c} on $u$.

In virtue of  $0<p<1$, it follows   from equation \eqref{3.1} that for all $(x,t)\in B_1(Re_1)\tm[1,T],$
 \begin{eqnarray}\label{sbp2}
 (\partial_t-\Delta)^s u(x,t)&\geq& (R-1)u^p(x,t)\nonumber\\
 &\geq& (R-1)\fr{u(x,t)}{M^{1-p}},
 \end{eqnarray}
 where
\[M:=\sup\limits_{\R^n\tm\R}u(x,t).\]

  Now we choose sufficiently large  $R$ such that \be\label{sbp3}\fr{R-1}{M^{1-p}}>\fr{C_s}{m_T}+\lm_1.\ee

Let $$w(x,t)=u(x,t)-v(x,t),$$then by \eqref{sbp1}, \eqref{sbp2},\eqref{sbp3}, $w$ satisfies
\begin{equation}\label{sbp4}
\begin{cases}
(\partial_t-\Delta)^s w(x,t)\geq \lm_1 w(x,t),& \mb{ in } B_1(Re_1)\tm(1,T], \\
w(x,t)
>0, &\mb{ in } (\R^n\tm(-\infty,T])\backslash (B_1(Re_1)\tm(1,T]).
\end{cases}
\end{equation}

Consequently, by the  maximum principle for the operator $(\partial_t-\Delta)^s-\lm_1$ (Theorem \ref{thm12}), we conclude

\be\label{sbp5}w(x,t)>0,\,\mb{for~all}\,(x,t)\in B_1(Re_1)\tm(1,T],\ee
that is,
\[u(x,t)>\phi_R(x)(t^\beta-1) \,\,\mb{in}\,\,B_1(Re_1)\tm(1,T].\]
Therefore,

\[M>\sup\limits_{\R^n}\phi_R(x)(T^\beta-1),\]
which induces a contradiction.
This  shows that equation  \eqref{3.1} possesses no positive bounded solution and  thus completes the proof of Theorem\ref{thm4}.

\section{Acknowledgments} \,
The work of the first author is partially supported by MPS Simons foundation 847690 and National Natural Science Foundation of China (NSFC Grant No. 12071229).  The work of the second author is partially supported by the National Natural Science Foundation of China (Grant No.12501145, W2531006, 12250710674 and 12031012), the Natural Science Foundation of Shanghai (No. 25ZR1402207),   the China Postdoctoral Science Foundation (No. 2025T180838 and 2025M773061), the Postdoctoral Fellowship Program of CPSF (No. GZC20252004), and the Institute of Modern Analysis-A Frontier Research Center of Shanghai.

\end{proof}
\bigskip

Wenxiong Chen

Department of Mathematical Sciences

Yeshiva University

New York, NY, 10033, USA

wchen@yu.edu
\medskip

Yahong Guo

School of Mathematical Sciences

Shanghai Jiaotong University

Shanghai, 200240, P.R. China and
\smallskip

School of Mathematical Sciences

Nankai University

Tianjin, 300071 P. R. China 

guoyahong1995@outlook.com

\end{document}